\documentclass[11pt]{amsart}
\usepackage{amssymb}
\theoremstyle{plain}
\newtheorem{theorem}{Theorem}
\newtheorem{corollary}{Corollary}

 \newtheorem{lemma}{Lemma}

\theoremstyle{definition}

\theoremstyle{remark}

\numberwithin{equation}{section}

\setbox0=\hbox{$+$}
\newdimen\plusheight
\plusheight=\ht0
\def\+{\;\lower\plusheight\hbox{$+$}\;}

\setbox0=\hbox{$-$}
\newdimen\minusheight
\minusheight=\ht0
\def\-{\;\lower\minusheight\hbox{$-$}\;}

\setbox0=\hbox{$\cdots$}
\newdimen\cdotsheight
\cdotsheight=\plusheight
\def\cds{\lower\cdotsheight\hbox{$\cdots$}}

\begin{document}
\title[Divergence in the General Sense of $q$-Continued Fraction ]
       {On the Divergence in the General Sense of $q$-Continued Fraction on
the Unit Circle. }
\author{D. Bowman}
\address{Mathematics Department\\
       University of Illinois \\
        Champaign-Urbana, Illinois 61820}
\email{bowman@math.uiuc.edu}
\author{J. Mc Laughlin}
\address{Mathematics Department\\
       University of Illinois \\
        Champaign-Urbana, Illinois 61820}
\email{jgmclaug@math.uiuc.edu}
\thanks{The second author's research supported in part by a
Trjitzinsky Fellowship.}
\date{May, 11, 2001}
\begin{abstract}
We show, for each $q$-continued fraction $G(q)$ in a certain class
of continued fractions,
 that
there is an uncountable set of
points
 on the unit circle at which $G(q)$ diverges
in the general sense. This class  includes the Rogers-Ramanujan continued
fraction and the three Ramanujan-Selberg continued fraction.

We discuss the implications of our theorems for the general convergence of
other $q$-continued fractions, for example the
G\"{o}llnitz-Gordon continued fraction, on the unit circle.
\end{abstract}

\maketitle

\section{Introduction}
In  \cite{BML01}, we made a detailed study of the convergence behaviour of
the famous Rogers-Ramanujan continued fraction $K(q)$, where
{\allowdisplaybreaks
\begin{equation}\label{rreq}
 K(q):= 1
\+
 \frac{q}{1}
\+
 \frac{q^{2}}{1}
\+
 \frac{q^{3}}{1}
\+\,\cds.
\end{equation}
}

It is an easy consequence of Worpitzky's Theorem
 (see \cite{LW92}, pp. 35--36) that
$R(q)$ converges to a value in $\hat{\mathbb{C}}$
for any $q$ inside the
unit circle.
\begin{theorem}(Worpitzky) Let the continued fraction
$K_{n=1}^{\infty}a_{n}/1$
be such that $|a_{n}|\leq 1/4$ for $n \geq 1$. Then
$K_{n=1}^{\infty}a_{n}/1$ converges. All approximants of
the continued fraction lie in the disc $|w|<1/2$ and the value
of the continued fraction is in the disk $|w|\leq1/2$.
\end{theorem}

Suppose $|q|>1$.
For $n \geq 1$, define
{\allowdisplaybreaks
\[
K_{n}(q):=1+
\frac{q}{1}
\+
 \frac{q^{2}}{1}
\+
 \frac{q^{3}}{1}
\+ \cds \+
 \frac{q^{n}}{1}.
\]
}
Then
{\allowdisplaybreaks
\begin{align*}
\lim_{j \to \infty}K_{2 j + 1}(q) &= \frac{1}{K(-1/q)},\\
\lim_{j \to \infty}K_{2 j}(q) &= \frac{K(1/q^{4}) }{q}.
\end{align*}
}
This was stated by Ramanujan without proof and proved by Andrews, Berndt,
Jacobson and Lamphere in 1992 \cite{ABJL92}.

This leaves the question of convergence on the unit circle.
The convergence behaviour at roots of unity was investigated by
Schur, who showed in \cite{S17} that if $q$ is a primitive $m$-th
root of unity, where $m \equiv 0$ $(\text{mod}\,\,5)$, then $K(q)$
diverges and if
 $q$ is a primitive $m$-th
root of unity, $m \not \equiv 0 (\text{mod}\,\,5)$, then $K(q)$ converges
and
{\allowdisplaybreaks
\begin{equation}\label{E:ScM}
K(q) =
\lambda \, q^{(1 - \lambda \sigma  m )/5} K(\lambda),
\end{equation}
}
where $\lambda = \left( \frac{m}{5} \right)$ (the Legendre symbol)
and $\sigma$ is the least positive residue of $m \,(\text{mod}\,\, 5)$.
Note that
$K(1)= \phi = (\sqrt{5}+1)/2$, and $K(-1) = 1/ \phi$.

Remark: Schur's result was essentially proved by Ramanujan, probably earlier than
Schur (see \cite{R57}, p.383). However, he made a calculational error (see \cite{Hg97},
p.56).

There remains the question of whether
the Rogers-Ramanujan  continued fraction converges or diverges
at a point on the unit circle which is not a root of unity.
The chief difficulty in
trying to apply the usual convergence/divergence tests stems from
the facts that the
Rogers-Ramanujan continued fraction  converges at a set of
points that is dense on the unit circle and  diverges at another
such dense set. This is clear from the result of Schur  above.

This question about
convergence on the unit circle at points which were not roots of unity
remained unanswered  until our paper,   \cite{BML01},
where we showed the existence of
an uncountable set of points on the unit circle at which the
Rogers-Ramanujan continued fraction diverged.

To discuss this
topic we use the following notation. Let the regular continued
fraction expansion of any irrational $t  \in (0, 1)$ be denoted by
$t=[0, e_{1}(t), e_{2}(t),  \cdots]$. Let the $i$-th approximant
of this continued fraction expansion be denoted by
 $c_{i}(t)/d_{i}(t)$. We will sometimes write $e_{i}$ for $e_{i}(t)$,  $c_{i}$ for $c_{i}(t)$ etc,
if there is no danger of ambiguity. Let $ \phi = ( \sqrt{5}+1)/2$.
In    \cite{BML01}, we proved the following theorem.
 \begin{theorem} \label{T:t1}\cite{BML01}
Let
{ \allowdisplaybreaks
 \begin{align} \label{E:seq}
S&= \{t  \in (0, 1): e_{i+1}(t)  \geq   \phi^{d_{i}(t)} \text{
infinitely often} \}.
 \end{align}
}
Then $S$ is an uncountable set of measure zero and,  if $t  \in S$
and $y =  \exp (2  \pi i t)$,  then the Rogers-Ramanujan continued
fraction diverges at $y$.
 \end{theorem}
 We were also able to give explicit examples of points $y$ on the unit circle at
which $K(y)$ diverges.
 \begin{corollary} \label{C:ex}
Let $t$ be the number with continued fraction expansion equal $[0,
e_{1}, e_{2},   \cdots]$,  where $e_{i}$ is the integer consisting
of a tower of $i$ twos with an $i$ an top.
{ \allowdisplaybreaks
 \begin{multline*}
t=[0, 2, 2^{ \displaystyle{2^{2}}}, 2^{ \displaystyle{2^{
\displaystyle{2^{3}}}}},  \cdots]= \\
0.484848484848484848484848484848484848484848484848484848484 \\
84848484848484848484849277885083112437522992318812011  \cdots
 \end{multline*}
}
If $y =  \exp (2  \pi i t)$ then $K(y)$ diverges.
 \end{corollary}
We were also able to show the existence of an uncountable set of points
on the unit circle at which $R(q)$ diverges in the \emph{general} sense
(see below for the definition of \emph{general convergence}) and to give
explicit examples of such points (The point $y$ of Corollary \ref{C:ex}
is such a point, for example).

In \cite{BML03} we generalised Theorem \ref{T:t1}
 to a wider class
of $q$-continued fractions, a class which includes
the Rogers-Ramanujan continued fraction and the three
``Ramanujan-Selberg''
 continued fractions studied by Zhang in \cite{Z91}:
{\allowdisplaybreaks
\begin{align}\label{z1}
S_{1}(q):= 1 + \frac{q}{1}
\+
 \frac{q+q^{2}}{1}
\+
 \frac{q^{3}}{1}
\+
 \frac{q^{2}+q^{4}}{1}
\+
\cds ,
\end{align}
}
{\allowdisplaybreaks
\begin{align}\label{z2}
S_{2}(q):=
1 +
 \frac{q+q^{2}}{1}
\+
 \frac{q^{4}}{1}
\+
 \frac{q^{3}+q^{6}}{1}
\+
 \frac{q^{8}}{1}
\+
\cds ,
\end{align}
}
and
{\allowdisplaybreaks
\begin{align}\label{z3}
S_{3}(q):=
1 +
 \frac{q+q^{2}}{1}
\+
 \frac{q^{2}+q^{4}}{1}
\+
 \frac{q^{3}+q^{6}}{1}
\+
 \frac{q^{4}+q^{8}}{1}
\+
\cds .
\end{align}
}
These continued
 fractions were  first studied by Ramanujan  \cite{R57}.
As a corollary to our theorem in \cite{BML03},
we were able to show, for each of the continued fractions above,
 the existence of an uncountable set of points on the unit circle
 at which the continued fraction diverged.

In this present paper we extend our result in \cite{BML01} on the divergence
in the \emph{general}
sense of the Rogers-Ramanujan continued fraction on the unit circle to a wider
class of $q$-continued fractions, a class which includes
$K(q)$,  $S_{1}(q)$,  $S_{2}(q)$  and $S_{3}(q)$.
We show that each of these $q$-continued fractions diverges in the
general sense at an uncountable set of points on the unit circle.

\section{Divergence in the General Sense of $q$-Continued Fractions
on the Unit Circle}

In \cite{J86}, Jacobsen revolutionised the subject of the convergence
of continued fractions by introducing the concept of \emph{general convergence}.
General convergence is defined, see \cite{LW92}, as follows.

Let the $n$-th approximant of the continued fraction
{\allowdisplaybreaks
\begin{equation*}
M=b_{0}+\cfrac{a_{1}}{b_{1} + \cfrac{a_{2}}{b_{2} + \cfrac{a_{3}}{b_{3}+
\cds }}}
\end{equation*}
}
be denoted by $A_{n}/B_{n}$ and let
\[S_{n}(w)= \frac{A_{n}+wA_{n-1}}{B_{n}+wB_{n-1}}.
\]
Define the chordal metric $d$ on $\hat{\mathbb{C}}$ by
\begin{equation}\label{E:d}
d(w,z)=\frac{|z-w|}
{\sqrt{1+|w|^{2}}\sqrt{1+|z|^{2}}}
\end{equation}
when $w$ and $z$ are both finite, and
\[
d(w, \infty) = \frac{1}{\sqrt{1+|w|^{2}}}.
\]
\textbf{Definition:} The continued fraction
$M$ is said to \emph{converge generally} to
$f \in \hat{\mathbb{C}}$ if there exist sequences
$\{v_{n}\}$, $\{w_{n}\} \subset \hat{\mathbb{C}}$ such that
$\liminf d(v_{n},w_{n})>0$ and
\[
\lim_{n \to \infty}S_{n}(v_{n})=\lim_{n \to \infty}S_{n}(w_{n}) = f.
\]
Remark: Jacobson shows in \cite{J86} that, if a continued fraction converges in
the general sense, then the limit is unique.

The idea of general convergence is of great significance because
classical convergence implies general
convergence (take $v_{n}=0$ and $w_{n}= \infty$, for all $n$),
but the converse does not necessarily hold.
General convergence is a natural extension
of the concept of classical convergence for continued fractions.

 We  consider continued fractions of the form
\begin{align}\label{E:cf1}
G(q):=b_{0}(q) +&
K_{n=1}^{\infty}\frac{a_{n}(q)}{b_{n}(q)}\\
:=g_{0}(q^{0})+&
\frac{f_{1}(q^{0})}{g_{1}(q^{0})}
 \+\,\cds \+
\frac{f_{k-1}(q^{0})}{g_{k-1}(q^{0})}
\+
\frac{f_{k}(q^{0})}{g_{0}(q^{1})} \notag \\
\+
&\frac{f_{1}(q^{1})}{g_{1}(q^{1})}
 \+\,\cds \+
\frac{f_{k-1}(q^{1})}{g_{k-1}(q^{1})}
\+
\frac{f_{k}(q^{1})}{g_{0}(q^{2})} \+\notag \\
  \cds \+
\frac{f_{k}(q^{n-1})}{g_{0}(q^{n})}
\+
&\frac{f_{1}(q^{n})}{g_{1}(q^{n})}
 \+\,\cds \+
\frac{f_{k-1}(q^{n})}{g_{k-1}(q^{n})} \+
\frac{f_{k}(q^{n})}{g_{0}(q^{n+1})} \+\cds \notag
\end{align}
where $f_{s}(x), g_{s-1}(x) \in \mathbb{Z}[q][x]$, for $1 \leq s
\leq k$. Thus, for $n\geq 0$ and $1 \leq s \leq k$,
\begin{align}\label{con4ab}
&a_{nk+s}=a_{nk+s}(q)=f_{s}(q^{n}),& &b_{nk+s-1}=b_{nk+s-1}(q)=g_{s-1}(q^{n}).&
\end{align}

Many well-known $q$-continued fractions, including
the Rogers-Ramanujan continued fraction, the three
Ramanujan-Selberg continued fractions studied by Zhang in \cite{Z91}, and the
G\"{o}llnitz-Gordon continued fraction,
\begin{equation}\label{ggcf}
GG(q):=1+q +
 \frac{q^{2}}{1+q^{3}}
\+
 \frac{q^{4}}{1+q^{5}}
\+
 \frac{q^{6}}{1+q^{7}}
\+\,\cds,
\end{equation}
have the form of the continued fraction at \eqref {E:cf1}, with $k$ at most 2.
It seems natural to consider a class of continued fractions which, in a sense,
contains all of the above continued fractions.

For the remainder of the paper $P_{n}(q)/$ $Q_{n}(q)$
denotes the $n$-th approximant of $G(q)$,
$P_{n}/Q_{n}$ if there is no danger of
ambiguity.
For later use, we recall some basic facts about continued fractions.
It is well known (see, for example, \cite{LW92}, p.9)
that the $P_{n}$'s and $Q_{n}$'s
satisfy the following recurrence relations.
{\allowdisplaybreaks
\begin{align}\label{recurrel}
P_{n}&=b_{n}P_{n-1}+a_{n}P_{n-2},\\
Q_{n}&=b_{n}Q_{n-1}+a_{n}Q_{n-2}.\notag
\end{align}
}
It is also well known (see also \cite{LW92}, p.9) that, for $n \geq 1$,
{\allowdisplaybreaks
\begin{align}\label{reclema}
P_{n}Q_{n-1}-P_{n-1}Q_{n-1} &= (-1)^{n-1}\prod_{i=1}^{n}a_{n}.
\end{align}
}
 Condition \ref{con4ab}
 also implies that if $q$ is a
primitive $m$-th root of unity
then $G(q)$ is periodic with period $m\,k$.
Indeed, if $q$ is a
primitive $m$-th root of unity and $j \geq 0$,
\begin{align}\label{perqb}
a_{jmk +r} &= f_{r}(q^{jm}) = f_{r}(q^{0}) = a_{r},
\\
&\phantom{as} \notag \\
b_{jmk +r} &= g_{r}(q^{jm}) = g_{r}(q^{0}) = b_{r}.\notag
\end{align}
We now assume certain facts about the approximants of $G(q)$,
and the convergence behaviour of $G(q)$, at certain roots of unity.

 We assume that there is a positive
integer $d$ and an integer
$s \in \{1,2,$ $\ldots,$ $ d \}$,
such that if
$m \equiv s \mod{d}$  and $(r,m)=1$, then
\begin{align}\label{cond1}
 q=\exp\left(2 \pi i r/m\right)
&\Longrightarrow
\begin{cases}
a_{n}(q) \not = 0, \text{ for } n \geq 1,\\
G(q) \text{  converges and } G(q) \not = 0.
\end{cases}
\end{align}
This integer $s$ will be referred to frequently in what follows.

We further assume that if $G(q)$ converges at
$q =   \exp\left(2 \pi i r/m\right)$, a primitive $m$-th root of unity,
then $G(q)$ converges at any $q' =   \exp\left(2 \pi i r'/m'\right)$,
a primitive $m'$-th root of unity, where $m \equiv m' (\text{mod } d)$
and $r \equiv r' (\text{mod } d)$.

We also
assume that there exists $\eta \in \mathbb{Q} $ such that if
 $H(q):= q^{\eta}/G(q)$ and $G(q)$ converges at
$q =   \exp\left(2 \pi i r/m\right)$
then
\begin{align}\label{cond3}
H\left(\exp\left(2 \pi i r/m\right)\right)
&= H\left(\exp\left(2 \pi i\, r'/
m'\right)\right),
\end{align}
with $r'$ and $m'$ as above.
Note that the above condition implies that $H(q)$ takes only
finitely many values
 at roots of unity. Let these values be denoted
$H_{1}$, $H_{2}$,$\ldots, H_{N_{G}}$.

We assume that for  all
$m \equiv s$ (mod $d$) that
there are  integers $K_{0}$,  $K_{1}$,
$K_{2}$,  $K_{3}$ and $K_{4}$,
depending only on $s$, such that
  {\allowdisplaybreaks
\begin{align}\label{cond6c}
a_{km}(q) &= K_{0}, \phantom{as}
    &P_{km-1}(q) = K_{1},
        &\phantom{Q_{km-1} } \\
Q_{km-2}(q) &= K_{2},
    \phantom{Q_{km-1} }
        &Q_{km-1}(q)P_{km-2}(q) = K_{3},
&\phantom{Q_{km-1} }\notag \\
|Q_{km-1}(q)|&=|P_{km-2}(q)| = K_{4}\notag
\end{align}
}
Here $k$ is the positive integer in the definition of the continued fraction
$G(q)$ at \eqref{E:cf1}.

 Finally, it is also assumed  that there exists
$r \not = u \in \{0,1,\ldots,d-1\}$,
 such that
\begin{align}\label{cond7}
H\left(\exp\left(2 \pi i r/s\right)\right) = H_{a}
&\not = H_{b} =
H\left(\exp\left(2 \pi i u/s\right)\right),
\end{align}
for some $a$, $b \in \{1,\ldots, N_{G}\}$.

\vspace{30pt}

It may be instructive at this point to  show how these abstract
conditions above apply to a particular continued fraction. We let
$G(q)=K(q)$.

If we compare the continued fractions at \eqref{rreq} and \eqref{E:cf1}, it
is clear that we can take $k=1$, $g_{0}(x) \equiv 1$ and $f_{1}(x) \equiv x$ (giving
$b_{n}(q)=g_{0}(q^{n})=1$ and $a_{n}(q)=f_{1}(q^{n}) = q^{n}$).

 From Schur's paper \cite{S17} (or see Table \ref{Ta:t5}, which contains
the relevant information from \cite{S17}) we can take $d=5$ and $s=1$ and if
$q$ is a primitive $m$-th root of unity, $m \equiv 1  \mod{5}$, then
$K(q)$ converges, giving Condition \ref{cond1} above.

 If we set $\eta = 1/5$ and set
$H(q) = q^{1/5}/K(q)$, we have from \eqref{E:ScM} that, if $q$ is a primitive
$m$-th root of unity, $m \not \equiv 0  \mod{5}$,
then
\begin{equation}\label{heq}
H(q)= \frac{q^{( \lambda \sigma  m )/5} }{\lambda K(\lambda)},
\end{equation}
where $\lambda = \left( \frac{m}{5} \right)$ (the Legendre symbol)
and $\sigma$ is the least positive residue of $m \,(\text{mod}\,\, 5)$.
It follows  that $H(q)$ can take only
ten possible values at roots of unity.

 It is also clear from \eqref{heq} that Conditions \ref{cond3}
and \ref{cond7} are satisfied, since if $m \equiv 1  \mod{5}$
(so that $\lambda = \sigma = 1$ and $K(1) =(1+\sqrt{5})/2$ ) and
$q = \exp\left(2 \pi i r/m\right)$, with $(r,m)=1$, then
\[
H(q)=\frac{2\exp(2 \pi i\,r/5)}{1+\sqrt{5}}.
\]

If $q$ is a primitive $m$-th root of unity, it follows from \eqref {rreq}
and \eqref{cond6c} that $K_{0}=a_{m}(q)=q^{m}=1$. It follows from
Schur's paper \cite{S17} (or, once again, from Table \ref{Ta:t5})
that $K_{1}=P_{m-1}(q)=1$, $K_{2}=Q_{m-2}(q)=0$ and
$K_{3}=P_{m-2}(q)Q_{m-1}(q) =q^{(1-m)/5}q^{(m-1)/5}=1=K_{4}$.
Thus Condition \ref{cond6c}
is satisfied.

From the paper of Zhang \cite{Z91}, each of
 $S_{1}(q)$, $S_{2}(q)$ and $S_{3}(q)$ also
satisfy a set of  conditions of the form set out in
\eqref{cond1} to  \eqref{cond7}. The relevant details are found in
Table \ref{Ta:t5}.

\vspace{30pt}

As before, let the regular continued fraction expansion of an
irrational $t \in (0,1)$ be denoted by
$[0,e_{1},e_{2}, \ldots]$ and let the $n$-th approximant of this
continued fraction be denoted by
$c_{n}/d_{n}$. We prove the following theorem.

\begin{theorem}\label{tgen}
Let $G(q)$ have the form given by \eqref{E:cf1} and satisfy conditions
 \eqref{con4ab} and
 \eqref{cond1} --  \eqref{cond7}.

There exists an integer $N'$ and a strictly increasing
function $\gamma: \mathbb{N} \to \mathbb{N}$ such that
if  $t$ is any irrational in $(0,1)$ for which
there exist two subsequences of approximants
$\{c_{f_{n}}/d_{f_{n}}\}$ and $\{c_{g_{n}}/d_{g_{n}}\}$
  satisfying
\begin{align}\label{e:fieqb}
c_{f_{n}} &\equiv r(\text{mod}\,\,d), &c_{g_{n}} \equiv u(\text{mod}\,\,d),\\
d_{f_{n}} &\equiv s(\text{mod}\,\,d), &d_{g_{n}} \equiv s(\text{mod}\,\,d).
\notag
\end{align}
and
\begin{align}\label{E:rconb}
e_{h_{n}+1} > 2 \pi \gamma (k\,N'\,d_{h_{n}}^{2}),
\end{align}
for all $n$, where $h_{n} = f_{n}$ or $g_{n}$. Then
$H(\exp(2 \pi i t))$ does not converge generally.

 Let $S^{\diamond}$ denote the set of all
$t \in (0,1)$ satisfying \eqref{e:fieqb} and \eqref{E:rconb}
and set
\begin{align}\label{G:uconb}
Y_{G} = \{ \exp(2 \pi i t): t \in S^{\diamond} \}.\\
\phantom{as}\notag
\end{align}
Then $Y_{G}$ is an uncountable set of measure zero.
\end{theorem}
We show, as a corollary to this theorem,  for each of the continued fractions
$K(q)$, $S_{1}(q)$, $S_{2}(q)$ and $S_{3}(q)$, that
there exists an uncountable set of points on the unit circle at which the
continued fraction does not converge generally.

The main idea of the proof will be to show that there exist points $y$
on the unit circle for which there exist two sequences of positive integers,
$\{m_{i}\}$ and $\{n_{i}\}$, such that the
 subsequences of approximants to $H(y)$,
$\{P_{n_{i}}/Q_{n_{i}}\}$ and $\{P_{n_{i}-1}/Q_{n_{i}-1}\}$ each tend to
the same limit, $L_{1}$ say, and the subsequences
 $\{P_{m_{i}}/Q_{m_{i}}\}$ and  $\{P_{m_{i}-1}/Q_{m_{i}-1}\}$
each tend to the same limit $L_{2} \not =L_{1}$.
This is done by constructing real numbers $t$
in the interval $(0,1)$ whose continued fraction expansions have a certain
rapid convergence behavior and then setting $y = \exp(2 \pi it)$.
In addition, it is shown that the
sequences $\{Q_{n_{i}}/Q_{n_{i}-1}\}$ and $\{Q_{m_{i}}/Q_{m_{i}-1}\}$
are  bounded from above, for $i$ sufficiently large. These two conditions
are then shown to imply that $H(q)$ does not converge generally at $y$.

We first give some technical lemmas. The proofs are not given
if the results are well known. Our aim is to estimate $P_{i}(q)$ and
$Q_{i}(q)$  for sequences of $i$'s in certain arithmetic progressions.
We use matrix notation since the proofs are
simpler.

\begin{lemma}\label{L:l1}
Let $G(q)$ be as in \eqref{E:cf1}.
There exist strictly increasing sequences of positive integers
$\{\kappa_{n}\}$ and $\{\nu_{n}\}$
  such that if $x$ and $y$ are any
two points on the unit circle then, for all integers  $n \geq 0$,
{\allowdisplaybreaks
\begin{align}\label{E:qdif}
&\phantom{asdaai} \notag\\
&|Q_{n}(x)-Q_{n}(y)| \leq \kappa_{n}|x-y|,
\end{align}
}
and
{\allowdisplaybreaks
\begin{align}\label{E:pdif}
|P_{n}(x)-P_{n}(y)| \leq \nu_{n}|x-y|.
\end{align}
}
\end{lemma}
\begin{proof}
Let $\{f_{n}(q)\}$ be any sequence of polynomials in $\mathbb{Z}[q]$.
Suppose $f_{n}(q) = \sum_{i=0}^{M_{n}} \gamma_{i}q^{i}$, where the
$\gamma_{i}$'s are in $\mathbb{Z}$. Then
{\allowdisplaybreaks
\begin{align*}
|f_{n}(x)-f_{n}(y)|&\leq
\sum_{i=1}^{M_{n}}|\gamma_{i}|\left|x^{i}-y^{i}\right|\\
&\phantom{asdaai}\notag \\
&\leq \sum_{i=1}^{M_{n}}i\,|\gamma_{i}||x-y|.
\end{align*}
}
Now set $\delta_{n} =
\max \left\{\,\, \sum_{i=1}^{M_{n}}i\,|\gamma_{i}|,\,\,\, 1,
\,\,\delta_{n-1}+1\right\}$.
Inequality \eqref{E:qdif} follows by setting $f_{n}(q)=Q_{n}(q)$
and $\delta_{n}=\kappa_{n}$.
The result for
\eqref{E:pdif}  follows similarly.
\end{proof}
With $\kappa_{n}$ and $\nu_{n}$ as in the above lemma, define,
for each $n \geq 1$,
\begin{align}\label{gameq}
\gamma(n) &:= \max \{\kappa_{n}, \nu_{n} \}.
\end{align}
This function will be used later in the proof of
Theorem \ref{tgen}.
\begin{lemma}\label{lem5}$($\cite{SVDP92}, p. 238$)$
For $n \geq 0$,
\begin{equation}
\left(
\begin{matrix}
P_{n} &P_{n-1}\\
\phantom{a} &\phantom{a}\\
Q_{n} &Q_{n-1}
\end{matrix}
\right)
=
\left(
\begin{matrix}
b_{0} & 1 \\
\phantom{a} &\phantom{a}\\
1 & 0
\end{matrix}
\right)
\prod_{i=1}^{n}
\left(
\begin{matrix}
b_{i} & 1 \\
\phantom{a} &\phantom{a}\\
a_{i} & 0
\end{matrix}
\right).
\end{equation}
\end{lemma}
\begin{proof}
This follows, by induction, from the recurrence relations
\eqref{recurrel}.
\end{proof}
We now assume that $q$ is a primitive $m$-th root of unity,
$q=\exp(2 \pi i n/m)$, where $(n,m)= 1$, $m \equiv s \mod{d}$ and either
$n \equiv r \mod{d}$ and $n \equiv u \mod{d}$,
 where $r$, $s$ and $u$ are as in
 condition \eqref{cond7}.
\begin{lemma}\label{lem6}
For $j \geq 1$ and $1 \leq r \leq km$,
\begin{multline}\label{rec1}
\left(
\begin{matrix}
P_{jkm+r} & P_{jkm-1+r}\\
\phantom{a} &\phantom{a}\\
Q_{jkm+r} &Q_{jkm-1+r}
\end{matrix}
\right)
\\
\phantom{as}\\
=
\left(
\begin{matrix}
P_{km-1} & a_{km}P_{km-2}\\
\phantom{a} &\phantom{a}\\
Q_{km-1} &a_{km}Q_{km-2}
\end{matrix}
\right)
\left(
\begin{matrix}
P_{(j-1)km+r} & P_{(j-1)km-1+r}\\
\phantom{a} &\phantom{a}\\
Q_{(j-1)km+r} &Q_{(j-1)km-1+r}
\end{matrix}
\right).
\end{multline}
For $j \geq 1$ and $0 \leq r \leq km -1$,
\begin{align}\label{rec2}
\left(
\begin{matrix}
P_{jkm+r} & \\
\phantom{a} \\
Q_{jkm+r}
\end{matrix}
\right)
&=
\left(
\begin{matrix}
P_{km-1} & a_{km}P_{km-2}\\
\phantom{a} &\phantom{a}\\
Q_{km-1} &a_{km}Q_{km-2}
\end{matrix}
\right)^{j}
\left(
\begin{matrix}
P_{r} \\
\phantom{a} \\
Q_{r}
\end{matrix}
\right).
\end{align}
\end{lemma}
\begin{proof}
By Lemma \ref{lem5} and the periodicity of the
$a_{i}$'s and/or the $b_{i}$'s noted at
 \eqref{perqb}, we have that
{\allowdisplaybreaks
\begin{align}
&\left(
\begin{matrix}
P_{jkm+r} & P_{jkm-1+r}\\
\phantom{a} &\phantom{a}\\
Q_{jkm+r} &Q_{jkm-1+r}
\end{matrix}
\right)
=
\left(
\begin{matrix}
b_{0} & 1 \\
\phantom{a} &\phantom{a}\\
1 & 0
\end{matrix}
\right)
\prod_{i=1}^{jkm+r}
\left(
\begin{matrix}
b_{i} & 1 \\
\phantom{a} &\phantom{a}\\
a_{i} & 0
\end{matrix}
\right)\notag \\
&\phantom{as} \notag \\
&=\left(
\begin{matrix}
b_{0} & 1 \\
\phantom{a} &\phantom{a}\\
1 & 0
\end{matrix}
\right)
\prod_{i=1}^{km}
\left(
\begin{matrix}
b_{i} & 1 \\
\phantom{a} &\phantom{a}\\
a_{i} & 0
\end{matrix}
\right)
\left(
\begin{matrix}
0 & 1 \\
\phantom{a} &\phantom{a}\\
1 & -b_{0}
\end{matrix}
\right)
\left(
\begin{matrix}
b_{0} & 1 \\
\phantom{a} &\phantom{a}\\
1 & 0
\end{matrix}
\right)
\prod_{i=km+1}^{jkm+r}
\left(
\begin{matrix}
b_{i} & 1 \\
\phantom{a} &\phantom{a}\\
a_{i} & 0
\end{matrix}
\right) \notag \\
&\phantom{as} \notag \\
&=\left(
\begin{matrix}
P_{km} & P_{km-1}\\
\phantom{a} &\phantom{a}\\
Q_{km} &Q_{km-1}
\end{matrix}
\right)
\left(
\begin{matrix}
0 & 1 \\
\phantom{a} &\phantom{a}\\
1 & -b_{km}
\end{matrix}
\right)
\left(
\begin{matrix}
b_{0} & 1 \\
\phantom{a} &\phantom{a}\\
1 & 0
\end{matrix}
\right)
\prod_{i=1}^{(j-1)km+r}
\left(
\begin{matrix}
b_{i} & 1 \\
\phantom{a} &\phantom{a}\\
a_{i} & 0
\end{matrix}
\right). \notag
\end{align}
}
Statement \eqref{rec1} then follows
from the facts that $P_{km} = b_{km}P_{km-1}+a_{km}P_{km-2}$ and
$Q_{km} = b_{km}Q_{km-1}+a_{km}Q_{km-2}$
and  Lemma \ref{lem5}. Statement \eqref{rec2} is an immediate
consequence of \eqref{rec1}.
\end{proof}
Remark: It is clear from \eqref{rec2},
 that if $G(q)$ converges then
$Q_{km-1} \not = 0$, since otherwise  $Q_{jkm-1}  = 0$ for $j \geq 1$.

Define
\begin{align}
M &:= \left(
\begin{matrix}
P_{km-1} & a_{km}P_{km-2}\\
\phantom{a} &\phantom{a}\\
Q_{km-1} &a_{km}Q_{km-2}
\end{matrix}
\right).
\end{align}
Equation \eqref{reclema} implies that
\begin{align}\label{detint}
Det(M)=
a_{km}(P_{km-1}Q_{km-2}-P_{km-2}Q_{km-1}) &= (-1)^{km}\prod_{i=1}^{km}a_{i}.
\end{align}
Let $T$ denote the trace of $M$ and $D$ its determinant.
In light of \eqref{detint} and \eqref{cond6c} it is clear that
$T$ and  $D$ are both integers that depend only on $s$.
From this it is clear that
\begin{align}\label{cond8a}
&T^{2}-4\,D = K_{5}, \text{ for some }K_{5} \in \mathbb{Z},
\end{align}
and that $K_{5}$ also depends only on $s$. The eigenvalues of $M$
are
\begin{align}\label{evals}
\lambda_{1} &= \frac{T+\sqrt{T^{2}-4\,D}}{2},\\
\lambda_{2} &= \frac{T-\sqrt{T^{2}-4\,D}}{2}. \notag
\end{align}
The corresponding eigenvectors are
{\allowdisplaybreaks
\begin{align}\label{evecs}
\textbf{x}:=
\left(
\begin{matrix}
x_{1}\\
\phantom{a}\\
x_{2}
\end{matrix}
\right)
&=
\left(
\begin{matrix}
\displaystyle{
\frac{P_{km-1}-a_{km}Q_{km-2}-\sqrt{T^{2}-4\,D}}{2Q_{km-1}}
}\\
\phantom{a}\\
1
\end{matrix}
\right)
\end{align}
}
and
\begin{align}\label{evecsb}
\textbf{y}:=
\left(
\begin{matrix}
y_{1}\\
\phantom{a}\\
y_{2}
\end{matrix}
\right)
&=
\left(
\begin{matrix}
\displaystyle{
\frac{P_{km-1}-a_{km}Q_{km-2}+\sqrt{T^{2}-4\,D}}{2Q_{km-1}}
}
\\ \phantom{a}\\
1
\end{matrix}
\right).
\end{align}

As shown above,  if $G(q)$ converges,
then $Q_{km-1} \not = 0$, and this justifies taking
$x_{2}=y_{2}=1$. Note for later use that $|x_{1}|$, $|y_{1}|$,
$\lambda_{1}$ and  $\lambda_{2}$ depend only on $s$.
This follows from \eqref{cond6c}
 and \eqref{cond8a}.
\begin{lemma}\label{lem8}
Let the eigenvalues of
\[
M = \left(
\begin{matrix}
P_{km-1} & a_{km}P_{km-2}\\
\phantom{a} &\phantom{a}\\
Q_{km-1} &a_{km}Q_{km-2}
\end{matrix}
\right)
\] be $\lambda_{1}$ and
$\lambda_{2}$. If $G(q)$ converges then $\lambda_{1}=\lambda_{2}$ or
$|\lambda_{1}| \not = |\lambda_{2}|$.
\end{lemma}
\begin{proof}
Since $a_{i} \not = 0$ for $i \geq 1$,
it follows from  \eqref{reclema},  that
$\text{Det}(M)
\not = 0$, so that
neither of the eigenvalues is zero.

Suppose $\lambda_{1} \not =
\lambda_{2}$ but $|\lambda_{1}| =
|\lambda_{2}|$. In this case it is clear from \eqref{evals}
and \eqref{evecs} that \textbf{x} and \textbf{y} are linearly
independent.
For $r \in \{0,1,2,\dots,km-1\}$, suppose that $(P_{r},Q_{r})^{T} = p_{r}$
\textbf{x}$+q_{r}$ \textbf{y}, for some $p_{r}$, $q_{r} \in \mathbb{C}$.
 Then it follows from \eqref{rec2}, \eqref{evecs} and \eqref{evecsb}   that
\begin{align}\label{mpqb}
\left(
\begin{matrix}
P_{jmk+r}\\
\phantom{a}\\
Q_{jmk+r}
\end{matrix}
\right)
=\left(
\begin{matrix}
 p_{r} \lambda_{1}^{j}x_{1}
 + q_{r}\lambda_{2}^{j}y_{1}\\
\phantom{a}\\
p_{r} \lambda_{1}^{j}
 + q_{r}\lambda_{2}^{j}
\end{matrix}
\right).
\end{align}
By some simple algebraic manipulation,
\begin{align}\label{lameq}
\frac{P_{jkm+r}}{Q_{jkm+r}}=y_{1} + \frac{p_{r}\,
     \left( x_{1}-y_{1}       \right) }
{p_{r}\, +
       q_{r} \,\left(\frac{\lambda_{2}}{\lambda_{1}}\right)^{j} }.
\end{align}
The right hand side does not  converge as $j \to \infty$,
unless $p_{r} = 0$ or $q_{r} = 0$,
for each $r$.

 Since we are considering the case where no $a_{i} = 0$,
then $a_{1} \not = 0$,
$(P_{0},Q_{0}) \not = \gamma (P_{1},Q_{1})$, for any
$\gamma \in \mathbb{C}$. Hence $p_{0}q_{1}-p_{1}q_{0} \not = 0$.

We first consider the case  $p_{0} = 0$.
Then $\lim_{j \to \infty}$ $P_{jkm}$ $/Q_{jkm}$ $ = $ $y_{1}$.
Since $p_{0} = 0$, it follows from the remark above that
 $p_{1}\not = 0$, and then it must be that
$q_{1}= 0$ and $\lim_{j \to \infty}P_{jkm+1}/Q_{jkm+1} = x_{1}
\not = y_{1} $,
which is  a contradiction.

On the other hand, if $p_{0} \not = 0$, then $q_{0} = 0$
and so $q_{1} \not = 0$ which necessitates $p_{1} =0$, and
a similar contradiction follows. This completes the proof.
\end{proof}

Remarks:  1) The eigenvalues for the Rogers-Ramanujan continued fraction and the
Ramanujan-Selberg continued fractions are non-zero and distinct.\\
2) It follows similarly
 from \eqref{lameq}, \eqref{evecs} and \eqref{evecsb},
 that, in the case $|\lambda_{1}| \not = |\lambda_{2}|$,
\begin{equation}
G(q)=
\begin{cases}
y_{1}
=
\displaystyle{
\frac{P_{km-1}-a_{km}Q_{km-2}+\sqrt{T^{2}-4\,D}}{2Q_{km-1}}
},\,\,\,\, |\lambda_{1}| < |\lambda_{2}|,\\
x_{1}=
\displaystyle{
\frac{P_{km-1}-a_{km}Q_{km-2}-\sqrt{T^{2}-4\,D}}{2Q_{km-1}}
},\,\,\,\,|\lambda_{1}| > |\lambda_{2}|.
\end{cases}
\end{equation}
For later use we evaluate $G(q)$  when
 $\lambda_{1}  =\lambda_{2}$. In this case $T^{2}-4\,D = 0$.
This equation implies
\begin{equation}\label{pqeq}
P_{km-2}=-\frac{(P_{km-1}-a_{km}Q_{km-2})^{2}}
{4\,a_{km}\,Q_{km-1}}.
\end{equation}
This in turn means that $P_{km-1}\not =a_{km}Q_{km-2}$, or else
$P_{km-2}=0$ and \eqref{rec2} gives that $P_{jkm-2}=0$ for $j\geq 1$,
implying that $G(q)=0$, contradicting our assumption.

For ease of notation we write $P_{km-1}=a$, $Q_{km-1}=c$ and
$a_{km}Q_{km-2} =d$. Then it follows from Lemma \ref{lem6} and  induction
that
{\allowdisplaybreaks
\begin{multline}\label{eqevals}
\left(
\begin{matrix}
P_{jkm-1} & a_{km}P_{jkm-2}\\
\phantom{a} &\phantom{a}\\
Q_{jkm-1} &a_{km}Q_{jkm-2}
\end{matrix}
\right)=\left(
\begin{matrix}
P_{km-1} & a_{km}P_{km-2}\\
\phantom{a} &\phantom{a}\\
Q_{km-1} &a_{km}Q_{km-2}
\end{matrix}
\right)^{j} \\
\phantom{as}  \\
=\frac{(a+d)^{j-1}}{2^{j+1}c}
\left(
\begin{matrix}
(j(a-d)+a+d)2\,c & -j(a-d)^{2}\\
\phantom{a} &\phantom{a}\\
4\,j\,c^{2} &(j(d-a)+a+d)2c
\end{matrix}
\right).
\end{multline}
}
\vspace{5pt}
From this and \eqref{rec2} it follows that, for $0 \leq r \leq km-1$,
{\allowdisplaybreaks
\begin{align*}
\frac{P_{jkm+r}}{Q_{jkm+r}}&=
\frac{(j(a-d)+a+d) 2\,c\,P_{r}-j(a-d)^{2}Q_{r}}
{4\,j\,c^{2}P_{r}+(j(d-a)+a+d )2\,c\,Q_{r}}\\
&=
\frac{((a-d)+(a+d)/j) 2\,c\,P_{r}-(a-d)^{2}Q_{r}}
{4\,c^{2}P_{r}+((d-a)+(a+d)/j )2\,c\,Q_{r}},
\end{align*}
}
and that
{\allowdisplaybreaks
\begin{align}\label{eqlim}
&G(q)=\lim_{j \to \infty}\frac{P_{jkm+r}}{Q_{jkm+r}}
=\frac{a-d}{2c}=\frac{P_{km-1}-a_{km}Q_{km-2}}{2Q_{km-1}}.
\end{align}
}

This holds whether or not  $2\, c\,P_{r}-(a-d)Q_{r}  = 0$,
for any $r \in\{0,1,\ldots,km-1\}$.

Note that   \eqref{cond6c}  implies that $|G(q)|$
depends only on $s$.

In the following lemmas the  cases of equal and unequal eigenvalues
are considered
separately, with Lemmas \ref{lem9} to \ref{lem10} dealing
with the case of equal eigenvalues.

 Define $G_{n}(q) := P_{n}(q)/Q_{n}(q)$ and
$H_{n}(q):=x^{\eta}/G_{n}(q)$, where $\eta$ is as defined at \eqref{cond3}.

   For the case
of equal eigenvalues, it follows from \eqref{evals} that
$T=P_{km-1}+a_{km}Q_{km-2} \not = 0$.
Note also that the conditions at  \eqref{cond6c} imply that
$P_{km-1}-a_{km}Q_{km-2}=K_{1}-K_{0}\,K_{2}$, a fixed integer
depending only on $s$.

In the following lemmas a sequence of positive integers
$N_{1},\ldots, N_{10}$ and  a sequence of constants
is defined. These integers and constants will depend only on
the constants $K_{0}$, $K_{1}$, $K_{2}$ and $K_{3}$
described at  \eqref{cond6c}. As such, they will
depend only on $s$. To avoid repetition throughout
the lemmas, we state here that these integers are chosen to
satisfy $N_{1}<N_{3}<N_{4}<N_{7}<N_{8}$
and $N_{2}<N_{5}<N_{6}<N_{9}<N_{10}$.
For Lemmas  \ref{lem9} to \ref{lem10}, we assume that
$\lambda_{1}  =\lambda_{2}$.
\begin{lemma}\label{lem9}
There exist positive constants $D_{1}$, $D_{2}$, $D_{3}$, $D_{4}$,
$D_{5}$ and  $D_{6}$,
each depending only on $s$,
 such that  if $j \geq 1$, then
\begin{align}\label{eqev1}
|Q_{jkm-1}|&= j\,D_{2}\,D_{1}^{j}.
\end{align}
There exists a  positive integer $N_{1}$,
 depending only on $s$, such that  if $j \geq N_{1}$, then
\begin{align}\label{eqev2}
D_{3}\,j\,D_{1}^{j}\leq |Q_{jkm-2}|
&\leq D_{4}\,j\,D_{1}^{j}
\end{align}
and
\begin{align}\label{eqev3}
D_{5}\leq
\left|\frac{Q_{jkm-1}}{Q_{jkm-2}}\right|\leq
D_{6}.
\end{align}
\end{lemma}

\begin{proof}
To prove \eqref{eqev1}, we first equate entries  at
\eqref{eqevals}, using the fact that
$\lambda_{1}=\lambda_{2}= (P_{km-1}+a_{km}Q_{km-2})/2$.
\begin{align*}
Q_{jkm-1}&= j\,Q_{km-1}\,\lambda_{1}^{j-1},
\end{align*}
\eqref{eqev1} follows upon setting $D_{1}=|\lambda_{1}|$ and
$D_{2} = |Q_{km-1}/\lambda_{1}|$, recalling the  conditions   at
\eqref{cond6c} and
 the facts noted at the end of \eqref{evecsb}.

Note for later use that, since $M$ has determinant equal to a
non-zero integer and  has two  equal eigenvalues, then $D_{1}\geq 1$
so that $\lim_{j \to \infty}|Q_{jkm-1}| = \infty$. Inequality \eqref{eqev3} then
implies $\lim_{j \to \infty}|Q_{jkm-1}| = \infty$ also.

Statement
\eqref{eqev2} follows similarly from comparing corresponding matrix elements
 at
\eqref{eqevals}, namely,
\begin{align*}
Q_{jkm-2} = \left(a_{km}Q_{km-2}-P_{km-1}+\frac{2\lambda_{1}}{j}\right)
\frac{j\lambda_{1}^{j-1}}{2a_{km}}.
\end{align*}
Set
\begin{align*}
D_{4} &= \frac{2|\lambda_{1}|+|a_{km} Q_{km-2} -P_{km-1}|}
{2|a_{km}\lambda_{1}|}.
\end{align*}
 Take $N_{1}$ large enough so that
$2|\lambda_{1}|/N_{1} < |a_{km} Q_{km-2} -P_{km-1}|$ and then set
\begin{align*}
D_{3} = \frac{|a_{km} Q_{km-2} -P_{km-1}|-2|\lambda_{1}|/N_{1}}
{2|a_{km}\lambda_{1}|}.
\end{align*}
Statement \eqref{eqev3} follows from \eqref{eqev1} and \eqref{eqev2}, by taking
$D_{5}= D_{2}/D_{4}$ and $D_{6}= D_{2}/D_{3}$.
\end{proof}

\begin{lemma}\label{lem99}
There exists  positive constants $D_{7}$, $D_{8}$, $D_{9}$,
$D_{10}$ and $D_{11}$
and  positive integers $N_{3}$ and $N_{4}$,
each depending only on
$s$, such that, if  $j \geq 1$, then
{\allowdisplaybreaks
\begin{align}\label{cn1}
 \left|G(q) - G_{jkm-1}(q)\right| =  \frac{D_{7}}{j},
\end{align}
}
if $j \geq N_{3}$, then
{\allowdisplaybreaks
\begin{align}\label{cn2}
\frac{D_{8}}{j}&\leq \left|G(q) - G_{jkm-2}(q)\right|\leq \frac{D_{9}}{j},
\end{align}
}
and if $j \geq N_{4}$ and $n = jkm-1$ or $jkm-2$ , then
{\allowdisplaybreaks
\begin{align}\label{flrn}
\frac{D_{10}}{j}&\leq \left|H(q) - H_{n}(q)\right|\leq \frac{D_{11}}{j}.
\end{align}
}
\end{lemma}
\begin{proof}
Equations \eqref{eqevals} and \eqref{eqlim} give that
{\allowdisplaybreaks
\begin{align*}
|G(q)& - G_{jkm-1}(q)|\\
&=
\left| \frac{P_{km-1}-a_{km}Q_{km-2}}{2Q_{km-1}}
-\frac{(j+1)P_{km-1}-(j-1)a_{km}Q_{km-2}}{2\,j\,Q_{km-1}}\right|\\
&=\left|\frac{P_{km-1}+a_{km}Q_{km-2}}{2\,j\,Q_{km-1}}\right|.
\end{align*}
}
Set
{\allowdisplaybreaks
\begin{align*}
D_{7}&= \left| \frac{P_{km-1}+a_{km}Q_{km-2}}{2Q_{km-1}}\right|.
\end{align*}
}
Note that $D_{7} \not = 0$, since $D_{7}=|\lambda_{1}/ Q_{km-1}|\not = 0$.
Next, from \eqref{eqevals} and \eqref{eqlim}
we find that
{\allowdisplaybreaks
\begin{align*}
&\left|G(q) - G_{jkm-2}(q)\right|\\
&\phantom{as}\\
& =
\left| \frac{P_{km-1}-a_{km}Q_{km-2}}{2Q_{km-1}}+
\frac{j(P_{km-1}-a_{km}Q_{km-2})^{2}}
{2\,Q_{km-1}((j+1)a_{km}Q_{km-2}-(j-1)P_{km-1})}\right|\\
&\phantom{as}\\
&=\left|\frac{-P_{km-1}^{2}+a_{km}^{2}Q_{km-2}^{2}}{2\,Q_{km-1}
(-(P_{km-1}+a_{km}Q_{km-2})/j+(P_{km-1}-a_{km}Q_{km-2})
)}\right|
\frac{1}{j}.
\end{align*}
}
Choose $N_{3}$ such that
$|(P_{km-1}+a_{km}Q_{km-2})/N_{3}|<|(P_{km-1}-a_{km}Q_{km-2})|$
and set
{\allowdisplaybreaks
\begin{align*}
D_{8}&= \frac{|-P_{km-1}^{2}+a_{km}^{2}Q_{km-2}^{2}|}
{2\,|Q_{km-1}|
(|P_{km-1}+a_{km}Q_{km-2}|+|P_{km-1}-a_{km}Q_{km-2}|)
}
\end{align*}
}
and
{\allowdisplaybreaks
\begin{align*}
D_{9}&= \frac{|-P_{km-1}^{2}+a_{km}^{2}Q_{km-2}^{2}|}
{2\,|Q_{km-1}|
(|P_{km-1}-a_{km}Q_{km-2}|-|P_{km-1}+a_{km}Q_{km-2}|/N_{3})
}.
\end{align*}
}
Note that neither $D_{8}$ or $D_{9}$ is zero, since
$\lambda_{1}=(P_{km-1}+a_{km}Q_{km-2})/2 \not = 0$
and $P_{km-1}-a_{km}Q_{km-2} \not = 0$ from the remark following
\eqref{pqeq}.

Let $n = jkm-1$ or $jkm-2$ and set $M^{'}=\max\{D_{7},D_{9}\}$
and $m^{'}=\min\{D_{7},D_{8}\}$. Choose $N_{4}$ such that
$|G(q)|>M^{'}/N_{4}$ (Recall that $|G(q)| \not = 0$ and is constant
for fixed $s$). Let $j \geq N_{4}$. Then
{\allowdisplaybreaks
\begin{align*}
 \left|H(q) - H_{n}(q)\right|&=
\left|\frac{q^{\eta}}{G(q)}-\frac{q^{\eta}}{G_{n}(q)}\right|
=\left|\frac{G(q)-G_{n}(q)}{G(q)G_{n}(q)}\right|.
\end{align*}
}
Then
{\allowdisplaybreaks
\begin{align*}
 \frac{m^{'}}{j\,|G(q)\,G_{n}(q)|} &\leq \left|H(q) - H_{n}(q)\right|
\leq \frac{M^{'}}{j\,|G(q)\,G_{n}(q)|}.
\end{align*}
}
By the definition of $M'$ and the choice of $N_{4}$, it follows that
{\allowdisplaybreaks
\begin{align*}
\frac{m^{'}}{j\,|G(q)|\,(|G(q)|+M^{'})}
&\leq\left|H(q) - H_{n}(q)\right|
\leq \frac{M^{'}}{j\,|G(q)|(|G(q)|-M^{'}/N_{4})}.
\end{align*}
}
Set
{\allowdisplaybreaks
\begin{align*}
D_{10} &= \frac{m^{'}}{|G(q)|\,(|G(q)|+M^{'})},
\hspace{15pt}
D_{11}=\frac{M^{'}}{|G(q)|(|G(q)|-M^{'}/N_{4})}.
\end{align*}
}
The constants $D_{10}$ and $D_{11}$ depend
only on $s$, since $|G(q)|$, $m'$, $M'$ and $N_{4}$ depend only on $s$ .
\end{proof}

\begin{lemma}\label{lem10}
 Let $y$ be  another point on the unit circle.
There exist  positive constants $D_{13}$, $D_{14}$
and $D_{15}$ and  positive integers
$N_{7}$ and $N_{8}$, each depending only on  $s$, such that if
$j \ge N_{7}$ and
 $n=jkm-1$ or $jkm-2$,
and
{\allowdisplaybreaks
\begin{align*}
&P_{n}(y)= P_{n}(q) + \epsilon_{1},\,\,\,\,
Q_{n}(y)= Q_{n}(q) + \epsilon_{2}, \,\,\,\,
\epsilon = \max \{|\epsilon_{1}|,|\epsilon_{2}|\} < 1/2,
\end{align*}
}
 then
{\allowdisplaybreaks
\begin{align}\label{E:kns}
\left| G_{n}(y) - G_{n}(q) \right| \leq \frac{D_{13} \epsilon}{j};
\end{align}
}
if $j \geq N_{8}$ and  the angle between $q$ and $y$
(measured from the origin)
 is less than
$\pi/(2 |\eta|)$,
then
{\allowdisplaybreaks
\begin{equation}\label{E:rns}
|H_{n}(y)-H_{n}(q)| < D_{14}|q-y| + D_{15}\frac{ \epsilon}{ j}
\end{equation}
}
and
{\allowdisplaybreaks
\begin{equation}\label{E:rnr}
|H_{n}(y) - H(q)| \leq  D_{14}|q-y| + D_{15}\frac{ \epsilon}{ j}
+\frac{D_{11}}{j}.
\end{equation}
}
\end{lemma}

\begin{proof}
Let
{\allowdisplaybreaks
\begin{align*}
D_{12} = \max \{D_{7}, D_{9}\},
\hspace{25pt}
D_{2}'= \min\{D_{2}, D_{3}\},
\hspace{25pt}
N_{7} \geq \left \lceil \frac{1}{D_{2}'} \right \rceil.
\end{align*}
}
Set $D_{13}' = 3 + |G(q)|+D_{12}/N_{7}$ and set $D_{13}= 2 D_{13}'/D_{2}'$.
Choose $N_{8}$ such that
{\allowdisplaybreaks
\begin{align*}
\min \left \{
|G(q)|-\frac{D_{12}}
         {N_{8}},
|G(q)|-
\displaystyle{
    \frac{D_{12}}
         {N_{8}}}
-\displaystyle{
\frac{D_{13}}
     {2N_{8}}}
\right  \} \geq \frac{|G(q)|}{2}.
\end{align*}
}
From the fact that $D_{1}\geq 1$ together with \eqref{eqev1} and
\eqref{eqev2}, it follows that
{\allowdisplaybreaks
\begin{align*}
|Q_{n}|&\geq D_{2}'\,j\,D_{1}^{j}\geq D_{2}'\,j.
\end{align*}
}
 Let $j \geq  N_{7}$.
Then $ |Q_{n}|-1/2 > D_{2}'\,j-1/2 \geq D_{2}'\,j/2$ and
{\allowdisplaybreaks
\begin{align*}
\left| G_{n}(y) - G_{n}(q) \right|
&=\left|\frac{P_{n}(y)}{Q_{n}(y)}-\frac{P_{n}(q)}{Q_{n}(q)}\right|
=\left|\frac{\epsilon_{1}Q_{n}(q)-\epsilon_{2}P_{n}(q)}
       {Q_{n}(q)(Q_{n}(q) + \epsilon_{2})}\right|\\
&\phantom{asd}\\
&\leq \frac{|\epsilon_{1}-\epsilon_{2}|}{|Q_{n}(q) + \epsilon_{2}|}
    +\frac{|\epsilon_{2}||P_{n}(q)-Q_{n}(q)|}
    {|Q_{n}(q)||Q_{n}(q)+\epsilon_{2}|}\\
&\phantom{asd}\\
&= \frac{|\epsilon_{1}-\epsilon_{2}|}{|Q_{n}(q) + \epsilon_{2}|}
    +\frac{|\epsilon_{2}||G_{n}(q)-1|}
    {|Q_{n}(q)+\epsilon_{2}|}\\
&\phantom{asd}\\
&\leq \frac{2\epsilon}{||Q_{n}(q)|-\epsilon|}
+\frac{\epsilon\left(|G(q)|+\displaystyle{
D_{12}/j}
+1\right)}{||Q_{n}(q)|-\epsilon|}\\
&\phantom{asd}\\
&=\frac{\epsilon\left(|G(q)|+\displaystyle{
D_{12}/j}
+3\right)}{||Q_{n}(q)|-\epsilon|}\\
&\phantom{asd}\\
&\leq\frac{D_{13}'\,\epsilon}
{||Q_{n}(q)|-1/2|}
\leq \frac{2 D_{13}'\,\epsilon}
{D_{2}'\,\,j}
=\frac{D_{13} \epsilon}{j}.
\end{align*}
}
Here we have used \eqref{cn1}, \eqref{cn2},  the bounds on
$\epsilon_{1}$ and $\epsilon_{2}$ and
  the inequality  relating $|Q_{n}|$ and  $j$ above.

Similarly, if $j \geq N_{8}$, then
{\allowdisplaybreaks
\begin{align*}
|H_{n}(y)-H_{n}(q)|
&=\left|\frac{y^{\eta}}{G_{n}(y)}-\frac{q^{\eta}}{G_{n}(q)}\right|\\
&\phantom{asdddsdsd}\\
&= \left|\frac{G_{n}(q)(y^{\eta}-q^{\eta})
+q^{\eta}(G_{n}(q)-G_{n}(y))}{G_{n}(q)G_{n}(y)}\right|\\
&\phantom{asdddsdsd}\\
&\leq \frac{2|\eta||q-y|}{|G_{n}(y)|}+
\frac{|G_{n}(q)-G_{n}(y)|}{|G_{n}(q)||G_{n}(y)|}\\
&\phantom{asdddsdsd}\\
&\leq \frac{2|\eta||q-y|}{\left||G_{n}(q)|-\displaystyle{
D_{13}\, \epsilon /j}\right|}+
\frac{\displaystyle{
D_{13}\, \epsilon/j}}
{|G_{n}(q)|\left||G_{n}(q)|-\displaystyle{
D_{13}\,\epsilon/j}\right|}.
\end{align*}
} Here we have used \eqref{E:kns} and the fact that the angle
between $q$ and $y$ (measured from the origin) is less than
$\pi/(2|\eta|)$ implies that $|y^{\eta}-q^{\eta}| \leq
2|\eta||q-y|$ (This last inequality follows since the stated bound
implies $(q/y)^{\eta}$ lies in the first or fourth quadrant and
the fact that in these quadrants, chordal distance from 1 is less
than arc distance, which in turn is less than
 twice the chordal distance).
From \eqref{cn1} and \eqref{cn2},
it follows that
{\allowdisplaybreaks
\begin{align}\label{angbnd}
|H_{n}(y)-H_{n}(q)|
&\leq
\frac{2|\eta||q-y|}
{\left||G(q)|-
\displaystyle{( D_{12}+D_{13}\, \epsilon) /j}\right|}\\
&+
\frac{\displaystyle{D_{13}\, \epsilon / j}}
{||G(q)|-D_{12}/j|\left||G(q)|-\displaystyle{
(D_{12}+D_{13}\,\epsilon)/j}\right|} \notag\\
&\leq
\frac{2|\eta||q-y|}
{\left||G(q)|-
\displaystyle{( D_{12}+D_{13}/2) /N_{8}}\right|} \notag\\
&+
\frac{\displaystyle{D_{13}\, \epsilon / j}}
{||G(q)|-D_{12}/N_{8}|\left||G(q)|-\displaystyle{
(D_{12}+D_{13}/2)/N_{8}}\right|} \notag\\
&\leq
\frac{4|\eta|}{|G(q)|}|q-y|+\frac{4\,D_{13}}{|G(q)|^{2}}\frac{\epsilon}{j}.
\end{align}
}
Set $D_{14}=\max\{4/|G(q)|,\,\, 4|\eta|/|G(q)|\}$ and
$D_{15}=4\,D_{13}/|G(q)|^{2}$.

Statement \eqref{E:rnr} follows from \eqref{E:rns} and \eqref{flrn}.
\end{proof}
In the following three lemmas we assume
$|\lambda_{1}| > |\lambda_{2}|$.
\begin{lemma}\label{lem9b}
There exist
positive constants $C_{1}$, $C_{2}$, $C_{3}$, $C_{4}$, $C_{5}$,
 $C_{6}$ and $C_{7}$,
and a positive integer $N_{2}$,
each  depending only on
$s$,  such that,
  if  $j \geq 1$, then
{\allowdisplaybreaks
\begin{align}\label{ceq1}
C_{2}\,C_{1}^{j}<|Q_{jkm-1}| < C_{3}\,C_{1}^{j};
\end{align}
}
and if $j \geq N_{2}$, then
{\allowdisplaybreaks
\begin{align}\label{ceq2}
C_{4}\,C_{1}^{j}\leq|Q_{jkm-2}| \leq
C_{5}\,C_{1}^{j}
\end{align}
}
and
{\allowdisplaybreaks
\begin{align}\label{ceq3}
C_{6}\leq \frac{|Q_{jkm-1}|}{|Q_{jkm-2}|} \leq C_{7}.
\end{align}
}
\end{lemma}
\begin{proof}
Let $\lambda_{1}$, $\lambda_{2}$, $x_{1}$ and $y_{1}$ be as defined at
\eqref{evals}, \eqref{evecs} and \eqref{evecsb}. Then
{\allowdisplaybreaks
\begin{align*}
&\left(
\begin{matrix}
P_{km-1} & a_{km}P_{km-2}\\
\phantom{a} &\phantom{a}\\
Q_{km-1} &a_{km}Q_{km-2}
\end{matrix}
\right)
=\left(
\begin{matrix}
x_{1} & y_{1}\\
\phantom{a} &\phantom{a}\\
1 & 1
\end{matrix}
\right)
\left(
\begin{matrix}
\lambda_{1} & 0\\
\phantom{a} &\phantom{a}\\
0 & \lambda_{2}
\end{matrix}
\right)
\left(
\begin{matrix}
x_{1} & y_{1}\\
\phantom{a} &\phantom{a}\\
1 & 1
\end{matrix}
\right)^{-1}.
\end{align*}
}
From Lemma \ref{lem6} it follows that
{\allowdisplaybreaks
\begin{align}\label{unmat}
\left(
\begin{matrix}
P_{jkm-1} & a_{km}P_{jkm-2}\\
\phantom{a} &\phantom{a}\\
Q_{jkm-1} &a_{km}Q_{jkm-2}
\end{matrix}
\right)
&=\left(
\begin{matrix}
x_{1} & y_{1}\\
\phantom{a} &\phantom{a}\\
1 & 1
\end{matrix}
\right)
\left(
\begin{matrix}
\lambda_{1}^{j} & 0\\
\phantom{a} &\phantom{a}\\
0 & \lambda_{2}^{j}
\end{matrix}
\right)
\left(
\begin{matrix}
x_{1} & y_{1}\\
\phantom{a} &\phantom{a}\\
1 & 1
\end{matrix}
\right)^{-1}\\
&\phantom{a}\notag\\
&=\frac{1}{x_{1}-y_{1}}
\left(
\begin{matrix}
x_{1}\lambda_{1}^{j}-y_{1}\lambda_{2}^{j} &
-x_{1}y_{1}(\lambda_{1}^{j}-\lambda_{2}^{j})\\
\phantom{a} &\phantom{a}\\
 \lambda_{1}^{j}-\lambda_{2}^{j}
& x_{1}\lambda_{2}^{j}-y_{1}\lambda_{1}^{j}
\end{matrix}
\right).\notag
\end{align}
}
Thus
{\allowdisplaybreaks
\begin{align*}
Q_{jkm-1}=\frac{1-(\lambda_{2}/\lambda_{1})^{j}}
{x_{1}-y_{1}}\lambda_{1}^{j}.
\end{align*}
}
Statement \eqref{ceq1} follows with $C_{1}=|\lambda_{1}|$ and
{\allowdisplaybreaks
\begin{align*}
C_{2}=\frac{1-\left|\lambda_{2}/\lambda_{1}\right| }
{\left|x_{1}-y_{1}\right|}
\hspace{25pt}
\text{ and }
\hspace{25pt}
C_{3}=\frac{1+\left|\lambda_{2}/\lambda_{1}\right| }
{\left|x_{1}-y_{1}\right|}.
\end{align*}
}
Note that, since $M$ has a non-zero integral determinant and
$|\lambda_{1}|>|\lambda_{2}|$, $C_{1} = |\lambda_{1}|>1$.

Similarly,
{\allowdisplaybreaks
\begin{align*}
Q_{jkm-2}=\frac{-y_{1}\lambda_{1}^{j}}{a_{km}(x_{1}-x_{2})}
\left(1-\frac{x_{1}}{y_{1}}\left(\frac{\lambda_{2}}{\lambda_{1}}
\right)^{j}\right).
\end{align*}
}
Choose $N_{2}$ large enough so that
\[\left|\frac{x_{1}}{y_{1}}\right|
\left|\frac{\lambda_{2}}{\lambda_{1}}\right|^{N_{2}}
<1
\]
 and then take
{\allowdisplaybreaks
\begin{align*}
C_{4}=\left| \frac{y_{1}}{a_{km}(x_{1}-y_{1})}\right|
\left(1-\left|\frac{x_{1}}{y_{1}}\right|
\left|\frac{\lambda_{2}}{\lambda_{1}}\right|^{N_{2}}\right)
\hspace{10pt}
\end{align*}
}
and
{\allowdisplaybreaks
\begin{align*}
C_{5}=\left| \frac{y_{1}}{a_{km}(x_{1}-y_{1})}\right|
\left(1+\left|\frac{x_{1}}{y_{1}}\right|
\left|\frac{\lambda_{2}}{\lambda_{1}}\right|\right).\phantom{ads}\\
\phantom{as}
\end{align*}
}
Note that equation \eqref{evecsb} and the fact that none of $a_{km}$, $Q_{km-1}$ and
$P_{km-2}$ is zero  ensure that $x_{1}$, $y_{1} \not = 0$,
and hence that $C_{4}$, $C_{5}\not = 0$.
Clearly,  for $j \geq N_{2}$,
\begin{align*}
C_{6}:=\frac{C_{2}}{C_{5}}
\leq \left|\frac{Q_{jkm-1}}{Q_{jkm-2}}\right|
&\leq
\frac{C_{3}}{C_{4}}=:C_{7}.
\end{align*}
Note that, by the remarks following \eqref{evecsb},  all of these constants depend
only on $s$. Note also that the condition $|\lambda_{1}|>1$ implies
 $\lim_{j \to \infty}|Q_{jkm-1}| = \lim_{j \to \infty}|Q_{jkm-2}| =\infty$.

\end{proof}

\begin{lemma}
There
exist positive constants $C_{8}<1$, $C_{9}$, $C_{10}$,
$C_{11}$, $C_{12}$, $C_{13}$ and $C_{14}$ and positive
integers $N_{5}$ and $N_{6}$,
each depending only on $s$, such that, if $j \geq 1$, then
\begin{align}\label{c9}
C_{9}\,C_{8}^{j}
&\leq \left|G(q) - G_{jkm-1}(q)\right|
\leq C_{10}\,C_{8}^{j};
\end{align}
if $j \geq N_{5}$, then
\begin{align}\label{c11}
C_{11}\, C_{8}^{j} &\leq \left|G(q) - G_{jkm-2}(q)\right|
\leq C_{12}\, C_{8}^{j};
\end{align}
and if $j \geq N_{6}$ and
$n = jkm-1$ or $jkm-2$, then
\begin{align}\label{c14}
C_{13}C_{8}^{j}\leq \left|H(q) - H_{n}(q)\right|
\leq C_{14}C_{8}^{j}.
\end{align}
\end{lemma}
\begin{proof}
From \eqref{unmat}
it can be seen that $G(q)$ converges to $x_{1}$  (and thus,
from \eqref{cond6c} and \eqref{evecs}, that
$|G(q)|$ depends only on $s$) so that
{\allowdisplaybreaks
\begin{align*}
\left|G(q) - G_{jkm-1}(q)\right|&=
\left|x_{1}
- \frac{x_{1}\lambda_{1}^{j}-y_{1}\lambda_{2}^{j}}
{ \lambda_{1}^{j}-\lambda_{2}^{j}}\right|
=\frac{|x_{1}-y_{1}|}
{\left|1-\left(\lambda_{2}/\lambda_{1}\right)^{j}\right|}
\left|\frac{\lambda_{2}}{\lambda_{1}}\right|^{j}.
\end{align*}
}
Set $C_{8}= |\lambda_{2}/\lambda_{1}|<1$,
$C_{9}=|x_{1}-y_{1}|/(1+C_{8})$  and
$C_{10}=|x_{1}-y_{1}|/(1-C_{8})$,  and \eqref{c9} follows.
Note that $x_{1} \not = y_{1}$ (else the eigenvalues would be equal),
 so that $C_{9}$ and $C_{10}$ are
non-zero.

Next, choose $N_{5}$ large enough so that
$C_{8}^{N_{5}}<|y_{1}/x_{1}|$, and consider $j \geq N_{5}$. Thus,
\begin{align*}
\left|G(q) - G_{jkm-2}(q)\right|&=
\left|x_{1}
- \frac{x_{1}y_{1}( \lambda_{1}^{j}-\lambda_{2}^{j})}
{-x_{1}\lambda_{2}^{j}+y_{1}\lambda_{1}^{j} }\right|
=\frac{|x_{1}-y_{1}|}
{\left|y_{1}/x_{1}-\left(\lambda_{2}/\lambda_{1}\right)^{j}\right|}
\left|\frac{\lambda_{2}}{\lambda_{1}}\right|^{j}.
\end{align*}
Set
\begin{align*}
C_{11}=\frac{|x_{1}-y_{1}|}{|y_{1}/x_{1}|+C_{8}},
\hspace{35pt}
&C_{12}=\frac{|x_{1}-y_{1}|}
{|y_{1}/x_{1}|-C_{8}^{N_{5}}},
&\phantom{as}
\end{align*}
and  \eqref{c11} follows.

Finally, let $n= jkm-1$ or $jkm-2$,
set $m^{'}=\min\{C_{9}, C_{11}\}$ and
 $M^{'}= \max\{C_{10}, C_{12}\}$. Choose $N_{6}$
such that $|G(q)|>M^{'}C_{8}^{N_{6}}$. Let $j \geq N_{6}$. Then
\begin{align*}
 \left|H(q) - H_{n}(q)\right|&=
\left|\frac{q^{\eta}}{G(q)}-\frac{q^{\eta}}{G_{n}(q)}\right|
=\left|\frac{G(q)-G_{n}(q)}{G(q)G_{n}(q)}\right|.
\end{align*}
From the definitions of  $m^{'}$ and $M^{'}$, and the choice of $N_{6}$,
 it follows that
\begin{align*}
\frac{m^{'}C_{8}^{j}}{|G(q)|\,\left(|G(q)|+M^{'}\right)}
&\leq\left|H(q) - H_{n}(q)\right|
\leq \frac{M^{'}C_{8}^{j}}{|G(q)|\left(|G(q)|-M^{'}C_{8}^{N_{6}}\right)}
\end{align*}
Set
\[
C_{13} = \frac{m^{'}}{|G(q)|\,\left(|G(q)|+M^{'}\right)},
\hspace{25pt}
C_{14} = \frac{M^{'}}{|G(q)|\left(|G(q)|-M^{'}C_{8}^{N_{6}}\right)},
\]
and \eqref{c14} follows.
\end{proof}

\begin{lemma}\label{lem10b}
Let $y$ be a another point on the unit circle.
There exist  positive constants $C_{15}$, $C_{16}$ and $C_{17}$ and
positive integers
$N_{9}$ and $N_{10}$, each depending only on  $s$, such that if $j \geq N_{9}$,
 $n=jkm-1$ or $jkm-2$.
and,
{\allowdisplaybreaks
\begin{align*}
&P_{n}(y)= P_{n}(q) + \epsilon_{1}, \hspace{5pt}
Q_{n}(y)= Q_{n}(q) + \epsilon_{2},
\text{ with }\epsilon = \max \{|\epsilon_{1}|,|\epsilon_{2}|\} < 1/2.
\end{align*}
}
 then
{\allowdisplaybreaks
\begin{align}\label{E:ekns}
\left| G_{n}(y) - G_{n}(q) \right| &
\leq \frac{C_{15}\,\epsilon}{C_{1}^{j}};
\end{align}
}
If  $j \geq N_{10}$,
 $n=jkm-1$ or $jkm-2$ and
the angle between $q$ and $y$ (measured from the origin)
 is less than
$ \pi /(2|\eta|)$, then
{\allowdisplaybreaks
\begin{align}\label{E:erns}
|H_{n}(y)-H_{n}(q)| &< C_{16}|q-y| + C_{17}\,\frac{\epsilon}{C_{1}^{j}}
\end{align}
}
and
{\allowdisplaybreaks
\begin{align}\label{E:ernr}
|H_{n}(y) - H(q)| \leq C_{16}|q-y| + C_{17}\,
\frac{\epsilon}{C_{1}^{j}}
+C_{14}\,C_{8}^{j}.\\
&\phantom{as}\notag
\end{align}
}
\end{lemma}
\begin{proof}
Let
{\allowdisplaybreaks
\begin{align*}
C_{10}' &= \max \{ C_{10}, C_{12}\},
\hspace{5pt}
C_{2}' = \min \{ C_{2}, C_{4} \} \text{   and   }
\hspace{5pt}
N_{9} \geq \frac{-\log(C_{2}')}{\log(C_{1})}.
\end{align*}
}
Set
$C_{15}' = 3 + |G(q)|+ C_{10}'\,C_{8}^{N_{9}}$ and  $C_{15} =2 C_{15}'/C_{2}'$.
Choose $N_{10}$ such that
{\allowdisplaybreaks
\begin{align*}
\min \left \{
|G(q)|- \displaystyle{C_{10}'C_{8}^{N_{10}} }, \,\,\,\,\,
|G(q)|-\displaystyle{C_{10}'C_{8}^{N_{10}} }
-\displaystyle{\frac{C_{15}}
                {2C_{1}^{N_{10}}}}
\right  \} \geq \frac{|G(q)|}{2}.
\end{align*}
}
Let $j \geq N_{9}$.
The inequalities at  \eqref{ceq1} and \eqref{ceq2} imply  that
 $|Q_{n}| \geq C_{2}'C_{1}^{j}$. The condition on
$N_{9}$ implies that, if $j \geq N_{9}$, then
$|Q_{n}|-1/2>C_{2}'\,C_{1}^{j}/2$. By similar reasoning to that used
in the proof of \eqref{E:kns}, we find that
{\allowdisplaybreaks
\begin{align*}
\left| G_{n}(y) - G_{n}(q) \right|
&\leq \frac{|\epsilon_{1}-\epsilon_{2}|}{|Q_{n}(q) + \epsilon_{2}|}
    +\frac{|\epsilon_{2}||G_{n}(q)-1|}
    {|Q_{n}(q)+\epsilon_{2}|}\\
&\phantom{asd}\\
&\leq \frac{2\epsilon}{||Q_{n}(q)|-\epsilon|}
+\frac{\epsilon\left||G(q)|+ C_{10}'\,C_{8}^{j}
+1\right|}{||Q_{n}(q)|-\epsilon|}\\
&\phantom{asd}\\
&\leq \frac{C_{15}'\,\epsilon}
{||Q_{n}(q)||-1/2|}
\leq \frac{2 C_{15}'\,\epsilon}
{C_{2}'\,\,C_{1}^{j}}
=\frac{C_{15}\,\epsilon}{C_{1}^{j}}.
\end{align*}
}
Here we have used \eqref{c9}, \eqref{c11},  the bounds on
$\epsilon_{1}$ and $\epsilon_{2}$ in the statement of the lemma
and the inequality relating $|Q_{n}|$
 and  $C_{1}^{j} $  above.

Let $j \geq N_{10}$. As in the case where $\lambda_{1}=\lambda_{2}$,
\begin{align*}
|H_{n}(y)-H_{n}(q)|
&\leq \frac{2|\eta||q-y|}{|G_{n}(y)|}+
\frac{|G_{n}(q)-G_{n}(y)|}{|G_{n}(q)||G_{n}(y)|}\\
&\phantom{asdddsdsd}\\
&\leq \frac{2|\eta||q-y|}{\left||G_{n}(q)|-\displaystyle{
C_{15}\, \epsilon /C_{1}^{j}  }\right|}+
\frac{\displaystyle{
 C_{15}\, \epsilon /C_{1}^{j}   }}
{|G_{n}(q)|\left||G_{n}(q)|-\displaystyle{
C_{15}\, \epsilon /C_{1}^{j}
}\right|}.
\\&\phantom{asdddsdsd}
\end{align*}
Here  we have used \eqref{E:ekns} and once again the fact that
 the angle between $q$ and $y$, measured form the origin, is less
than $\pi/(2|\eta|$) implies that
$|y^{\eta}-q^{\eta}| \leq 2|\eta||q-y|$ (See the
explanation before \eqref{angbnd}).  Using \eqref{c9} and \eqref{c11},
it follows that,
{\allowdisplaybreaks
\begin{align*}
|H_{n}(y)-H_{n}(q)|
&\leq
\frac{2|\eta||q-y|} {||G(q)|-
    \displaystyle{C_{10}'C_{8}^{j}}-
    \displaystyle{C_{15}\, \epsilon/ C_{1}^{j}           }|     }\\
&\phantom{as}+
    \frac{\displaystyle{C_{15}\, \epsilon /C_{1}^{j}}}
         {||G(q)|- \displaystyle{C_{10}'C_{8}^{j} }|
||G(q)|-\displaystyle{C_{10}'C_{8}^{j}}
    -\displaystyle{C_{15}\, \epsilon/C_{1}^{j}}  |}\\
&\leq
\frac{2|\eta||q-y|} {||G(q)|-
    \displaystyle{C_{10}'C_{8}^{N_{10}}}-
    \displaystyle{C_{15}\, \epsilon/ C_{1}^{N_{10}}          }|     }\\
&\phantom{as}+
    \frac{\displaystyle{C_{15}\, \epsilon /C_{1}^{j}}}
         {||G(q)|- \displaystyle{C_{10}'C_{8}^{N_{10}} }|
||G(q)|-\displaystyle{C_{10}'C_{8}^{N_{10}}}
    -\displaystyle{C_{15}/(2C_{1}^{N_{10}})}  |}\\
&\leq
\frac{4|\eta|}{|G(q)|}|q-y|+
\frac{4\,C_{15}}{|G(q)|^{2}}\frac{\epsilon}{C_{1}^{j}}.
\end{align*}
}
Set $C_{16}=\max\{4|\eta|/|G(q)|,\,\,4/|G(q)|\}$
and $C_{17}= 4\,C_{15}/|G(q)|^{2}$.
Statement \eqref{E:ernr} follows from \eqref{E:erns} and \eqref{c14}.
\end{proof}
\begin{lemma}\label{2limlem}
There exists an uncountable set of points on the unit circle such
that, if $y$ is one of these points, then there exist two increasing sequences
of integer,
$\{n_{i}\}_{i=1}^{\infty}$ and $\{m_{i}\}_{i=1}^{\infty}$ say,
such that
\begin{align*}
\lim_{i \to \infty}H_{n_{i}}(y) &
=\lim_{i \to \infty}H_{n_{i}-1}(y) = H_{a},\\
\lim_{i \to \infty}H_{m_{i}}(y) &
=\lim_{i \to \infty}H_{m_{i}-1}(y) = H_{b},
\end{align*}
for some $a$, $b \in \{1,2,\cdots, N_{G}\}$, where $a \not = b$.
\end{lemma}
\begin{proof}
If $\lambda_{1}=\lambda_{2}$, we set $N' = N_{8}$.
If $|\lambda_{1}|>|\lambda_{2}|$, we set $N' = N_{10}$.
With the notation of Theorem \ref{tgen},
let $t \in S^{\diamond}$ and
set $y= \exp (2 \pi i t)$. Let $c_{f_{n}}/d_{f_{n}}$ be one of the
infinitely many approximants in the continued fraction expansion
of $t$ satisfying \eqref{e:fieqb} and \eqref{E:rconb},
 and set $x_{n} = \exp (2 \pi i c_{f_{n}}/d_{f_{n}})$, so
that $x_{n}$ is a primitive $d_{f_{n}}$-th root of unity and
 $H(x_{n}) = H_{a}$. Let $\gamma (n)$ be as defined at
\eqref{gameq}. We use, in turn, the fact that chord length is
shorter than arc length,  a standard bound on the
absolute value of the difference between
a real number and an approximant in its continued fraction expansion,
and   \eqref{E:rconb}, we find that
\begin{align}\label{E:rxydif}
|x_{n}-y| &< 2\pi\, \left|t - \frac{c_{f_{n}}}{d_{f_{n}}}\right|
 < \frac{2\,\pi}{ d_{f_{n}}^{2}e_{f_{n}+1}}
<\frac{1}{ d_{f_{n}}^{2}\gamma(k\,N'\,d_{f_{n}}^{2})}.
\end{align}
Let $n'= k\,N'\,d_{f_{n}}^{2}-1$ or $ k\,N'\,d_{f_{n}}^{2} -2$.
By \eqref{E:qdif}, \eqref{E:pdif}, \eqref{gameq},
 and \eqref{E:rxydif} it follows that
\begin{align}
|P_{n'}(x_{n})-P_{n'}(y)|\leq \gamma(n')|x_{n}-y|
<  \frac{1}{d_{f_{n}}^{2}},
\end{align}
and similarly
\begin{align}\label{E:rpq1}
|Q_{n'}(x_{n})-Q_{n'}(y)|&\leq \frac{1}{d_{f_{n}}^{2} }.
\end{align}
If $\lambda_{1}=\lambda_{2}$, then
by \eqref{E:rnr}, with $\epsilon = 1/d_{f_{n}}^{2}$
and $j =N'\, d_{f_{n}}$ (so that $j \geq N_{8}$), we find that
\vspace{5pt}
\begin{align}\label{E:rrnrdif}
|H_{n'}(y)-H_{a}| &=|H_{n'}(y)-H(x_{n})|
\leq \frac{D_{14}}{d_{f_{n}}^{2}\gamma(k\,N'\,d_{f_{n}}^{2} )}
 +\frac{ D_{15} }{d_{f_{n}}^{3}\,N'  }
+\frac{D_{11}}{N'\, d_{f_{n}} }.
\end{align}
If $|\lambda_{1}| \not = |\lambda_{2}|$ then \eqref{E:ernr} similarly
implies that
\vspace{5pt}
\begin{align}\label{E:r2rnrdif}
|H_{n'}(y)-H_{a}| &=|H_{n'}(y)-H(x_{n})|\\
&\leq \frac{C_{16}}{d_{f_{n}}^{2}\gamma(k\,N'\,d_{f_{n}}^{2} )}
 +\frac{ C_{17} }{d_{f_{n}}^{2}C_{1}^{N'\,d_{f_{n}}}}
+C_{14}\,C_{8}^{N'\,d_{f_{n}}  }. \notag
\end{align}
Thus, in either case,
\begin{align}
\lim_{n \to \infty}
H_{k\,N'\, d_{f_{n}}^{2}-1}(y)=\lim_{n \to \infty}
H_{k\,N'\, d_{f_{n}}^{2}-2}(y) = H_{a}.
\end{align}
Similarly,
\begin{align}
\lim_{n \to \infty}
H_{k\,N'\, d_{g_{n}}^{2}-1}(y) =\lim_{n \to \infty}
H_{k\,N'\, d_{g_{n}}^{2}-2}(y) = H_{b}.
\end{align}
The set $S^{\diamond}$ is  uncountable  because
the conditions for membership require restrictions on
only infinitely many of the partial quotients. One can easily construct
a subset for which there is no restriction on a fixed infinite set of
partial quotient. For each set of choices of positive integers
for these partial quotients, one can choose other partial quotients
so that the conditions for membership of $S^{\diamond}$ are fulfilled.
Since the collection of all such continued fractions
 is uncountable, $S^{\diamond}$
is an uncountable set.
 Thus
$Y_{G} = \{ \exp(2 \pi i t): t \in S^{\diamond}\} $ is an uncountable set.
\end{proof}
Before proving Theorem \ref{tgen}, we show that $Y_{G}$ has measure
zero. We use the following lemma.

\begin{lemma}\label{L:l3aa}
\text{$($\cite{RS92}$ \text{ pp. 140--141})$ }
Let $f(m)>1$ for $m=1,2,\dots$, and suppose that
$\sum_{n=1}^{\infty}$ $1/f(m)$ $< \infty$.
Then the set $S^{*} =
\{ t \in (0,1): e_{m}(t) > f(m) \text{ infinitely often}\,\,\}$ has measure
zero.
\end{lemma}
Let $f(m) =  2 \pi \gamma (k\,N'\,m^{2})$, where $\gamma(n)$
is the function at \eqref{gameq}.
 Since $\gamma(j) \geq j$
for $j \geq 1$, it follows that $f(n) \geq2 \pi n^{2}$ and thus
that $\sum_{n=1}^{\infty}1/f(n)$  converges. Since, for the
regular continued fraction expansion of any real number, $d_{i}>i$
for $i \geq 4$, it follows that $d_{i}^{2} \geq (i+1)^{2}$ for $i
\geq 4$, and thus it is clear from \eqref{E:rconb} that the
elements in $S^{\diamond}$ satisfy $e_{m}(t) > f(m)$ infinitely
often. Hence $S^{\diamond}$ (and thus $Y_{G}$) is a set of measure
zero.

  Of course the actual set
of points on the unit circle at which $G(q)$ does not converge
generally might have measure larger than zero.

\emph{Proof of Theorem \ref{tgen}}.
Let $y$ be any point in $Y_{G}$, where $Y_{G}$ is  defined in the
proof of Lemma \ref{2limlem}, and
let $t$ be the irrational in $(0,1)$ for which
$y=\exp (2 \pi i t)$. $N'$ is defined in Lemma \ref{2limlem}.
If $\lambda_{1}=\lambda_{2}$, we set $N_{1/2} = N_{1}$.
If $|\lambda_{1}|>|\lambda_{2}|$, we set $N_{1/2} = N_{2}$.

Suppose $H(y)$ converges generally to $f \in \hat{\mathbb{C}}$ and that
$\{v_{n}\}$, $\{w_{n}\}$  are
two sequences such that
\[
\lim_{n \to \infty}\frac{P_{n}+v_{n}P_{n-1}}
     {Q_{n}+v_{n}Q_{n-1}}=
\lim_{n \to \infty}\frac{P_{n}+w_{n}P_{n-1}}
     {Q_{n}+w_{n}Q_{n-1}}=\frac{y^{\eta}}{f}:=g.
\]Suppose first that $|g| < \infty$.
By construction, there exist two infinite strictly increasing
 sequences of positive integers
$\{n_{i}\}_{i=1}^{\infty}$,
$\{m_{i}\}_{i=1}^{\infty}$ $\subset \mathbb{N}$ such that
\begin{align*}
L_{a}:=\,\,\,\,  \frac{y^{\eta}}{H_{a}}\,\,=\,\,
\lim_{i \to \infty}
\frac{P_{n_{i}}(y)}{Q_{n_{i}}(y)}\,\,\,\,  &=\,\,\,\,
\lim_{i \to \infty}\frac{P_{n_{i}-1}(y)}{Q_{n_{i}-1}(y)}\,\,\,\,
\end{align*}
and
\begin{align*}
L_{b}:=\,\,\,\,  \frac{y^{\eta}}{H_{b}}\,\,=\,\,
\lim_{i \to \infty}
\frac{P_{m_{i}}(y)}{Q_{m_{i}}(y)}\,\,\,\,  &=\,\,\,\,
\lim_{i \to \infty}\frac{P_{m_{i}-1}(y)}{Q_{m_{i}-1}(y)},
\end{align*}
for some $a \ne b$, where $a,b \in \{1,2,\cdots,N_{G}\}$.
Also by construction each $n_{i}$ has the form
$k\,N'\,d_{k_{i}}^{2}-1$, where $d_{k_{i}}$ is some denominator convergent
in the continued fraction expansion of $t$. A similar situation holds for each
$m_{i}$.
 It can be further  assumed that
$L_{a} \ne g$, since
$L_{a} \ne L_{b}$. For ease of notation write
\begin{align*}
 &P_{n_{i}}(y)\, =\,P_{n_{i}},
&Q_{n_{i}}(y)\, =\,Q_{n_{i}},\,\phantom{aas}\\
  &P_{n_{i}-1}(y)\, =\,P_{n_{i}-1},
&Q_{n_{i}-1}(y)\, =\,Q_{n_{i}-1}.
\end{align*}
Write
$P_{n_{i}}= Q_{n_{i}}(L_{a} + \epsilon_{n_{i}})$ and
$P_{n_{i}-1}= Q_{n_{i}-1}(L_{a} + \delta_{n_{i}})$, where
$\epsilon_{n_{i}} \to 0$ and
$\delta_{n_{i}} \to 0$ as $ i \to \infty $,
so that
\[
\frac{Q_{n_{i}}(L_{a}+ \epsilon_{n_{i}})+
w_{n_{i}}Q_{n_{i}-1}(L_{a} + \delta_{n_{i}})}
{Q_{n_{i}}+w_{n_{i}}Q_{n_{i}-1}}= g + \gamma_{n_{i}},
\]
where  $\gamma_{n_{i}} \to 0$
 as $i \to \infty$. Thus
\[
w_{n_{i}}+\frac{Q_{n_{i}}}{Q_{n_{i}-1}}
=\frac{Q_{n_{i}}}{Q_{n_{i}-1}} \times
\frac{\epsilon_{n_{i}}-\delta_{n_{i}}}
    {g-L_{a}+\gamma_{n_{i}}-\delta_{n_{i}}}.
\]
Because of  \eqref{eqev3} or  \eqref{ceq3},
 the fact that each $n_{i}$ has the form
$k\,N'\,d_{k_{i}}^{2}-1$, where $d_{k_{i}}$ is some denominator convergent
in the continued fraction expansion of $t$ and \eqref{E:rpq1},
 it follows that $Q_{n_{i}}/ Q_{n_{i}-1}$ is absolutely
bounded for $N'\,d_{k_{i}} > N_{1/2}$.
Therefore the right hand side of the last
equality tends to 0 as $i \to \infty$ and thus
\begin{align}\label{w's}
&w_{n_{i}}+Q_{n_{i}}/ Q_{n_{i}-1} \to 0\text{ as }n_{i} \to \infty.
\end{align}
Note  that $|w_{n_{i}}| < \infty$ for all $i$ sufficiently large, since
$|Q_{n_{i}}/ Q_{n_{i}-1}|<\infty$. Similarly,
\begin{align}\label{v's}
v_{n_{i}}+Q_{n_{i}}/ Q_{n_{i}-1} \to 0\text{ as }n_{i} \to \infty.
\end{align}
By the \eqref{w's}, \eqref{v's} and the triangle inequality,
\[
\lim_{i \to \infty}|v_{n_{i}}-w_{n_{i}}| = 0.
\]
Thus
\[
\liminf_{n \to \infty} d(v_{n},w_{n}) =0.
\]
Therefore $H(y)$ does not converge generally. The proof in the case
where $g$ is infinite is similar.

Since $Y_{G}$ is
  uncountable, this proves the theorem.
\begin{flushright}
$\Box$
\end{flushright}
Remark: Clearly $H(y) = y^{\eta}/G(y)$
converges generally if and only if $G(y)$ converges generally.

We have the following corollary to Theorem \ref{tgen}.
\begin{corollary}\label{cgen}
For each of the continued fractions $K(q)$,
$S_{1}(q)$, $S_{2}(q)$ and $S_{3}(q)$, there exists
an uncountable set of points on the unit circle at which the
continued fraction does not converge generally.
\end{corollary}
{\allowdisplaybreaks
\begin{table}[ht]
  \begin{center}
    \begin{tabular}{| c | c | c | c | c |}
    \hline
 $G(q)$      & $K(q)$   & $S_{1}(q)$     &$S_{2}(q)$      &$S_{3}(q)$ \\ \hline
             &      &        &        &  \\
$\eta$              & $1/5$     &$1/8$      &$1/2$            & $1/3$ \\
         &      &       &         &  \\
$(s,d)$           & $(1,5)$     &$(1,8)$        &$(1,8)$      & $(1,6)$ \\
         &      &       &         &  \\
$H(q)$ & $\frac{2\exp(2 \pi i\,r/5)}{1+\sqrt{5}}$
                & $\frac{1}{\sqrt{2}\exp(-\pi i\,r/4)}$
                        &$\frac{1}{(1+\sqrt{2})\exp(\pi i\,r)}$
                                  & $\frac{1}{2\exp(4\pi i\,r/3)} $ \\
             &      &       &         &  \\
$k$              & $1$      &$2$        &$2$              & $1$ \\
         &      &       &         &  \\
$f_{1},\cdots,f_{k}$& $x$       & $q x^{2}$, $q x+q^{2}x^{2}$ &$q x^{2}+q^{2}x^{4}$,$q^{4}x^{4}$& $x+x^{2}$ \\
             &      &       &         &  \\
$a_{km}$         & $1$          & $2$       &$1$              & $2$ \\
             &      &        &        &  \\
$P_{km-1}$       & $1$          & $2$           &$3$              & $1$ \\
             &      &       &         &  \\
$Q_{km-2}$       & $0$          & $1$       &$1$              & $0$ \\
             &      &       &         &  \\
$Q_{km-1}$       & $q^{(m-1)/5}$& $\epsilon q^{(m^{2}-1)/8}$ &$q^{(m-1)/2}$ & $q^{(m-1)/3} $ \\
             &      &       &         &  \\
$P_{km-2}$       & $q^{(1-m)/5}$& $\epsilon q^{(m-1)^{2}/8}$ &$q^{(m+1)/2}$ & $q^{(2m+1)/3}$ \\
\hline
    \end{tabular}
\phantom{asdf}\\
    \caption{}\label{Ta:t5}
    \end{center}
\end{table}
}
\begin{proof}
We use information contained in Table \ref{Ta:t5}.
In each case,
$q$ $=$ $\exp $ $(2 \pi i r/m)$,  a primitive $m$-th root of unity and
$m \equiv s \mod{d}$, where
$(s,d)$ is the
pair of integers from \eqref{cond1}.
$k$ is the integer and $f_{1},\cdots,f_{k}$ are the
polynomials from the definition of the continued fraction $G(q)$ at \eqref{E:cf1}.
$H(q):=q^{\eta}/G(q)$, where $\eta$
is the rational in row one of the table.

Row three gives the
value of $H(q)$, when $q$ $=$ $\exp $ $(2 \pi i r/m)$ as above. $a_{km}$ is the
$km$-th partial numerator in $G(q)$,
as defined at \eqref{E:cf1}.

The values in the first, third and
 last four rows come from the papers of  Schur (\cite{S17})
and Zhang (\cite{Z91}). The values of $a_{km}$ can be determined from the
continued fractions at \eqref{rreq} and \eqref{z1} -- \eqref{z3}.
For the last two entries in the $S_{1}(q)$ column, $\epsilon = (-1)^{(m-1)/4}$,
this notation being employed to make the table fit the width of the page.

We give the proof for
$S_{1}(q)$ only, since the proof for each of the other continued fractions
is almost identical.
 One can easily check that $S_{1}(q)$ has the form
given at \eqref{E:cf1} and satisfies the condition at \eqref{con4ab},
with $k=2$. From the table (or the paper of Zhang \cite{Z91}),
  $S_{1}(q)$ satisfies \eqref{cond1} with $d=8$ and $s=1$.
Likewise,  \eqref{cond3} is  satisfied with $\eta=1/8$. Conditions \eqref{cond6c}
are satisfied with $K_{0}=2$, $K_{1}=2$, $K_{2}=1$ and $K_{3}=K_{4}=1$
(when $m \equiv 1 \mod{8}$). It is clear from row three of the table that
\eqref{cond7} is satisfied, provided we choose $r \not \equiv u \mod{8}$.
The conditions required by the theorem are satisfied, and the result follows.
\end{proof}

\section{Concluding Remarks}
In proving the existence of an uncountable set of points on the unit
circle at which a $q$-continued fraction $G(q)$ does not converge in the
general sense,
our methods rely on knowing the behavior of the continued fraction
 at roots of unity and, if $q$ is a primitive
$m$-th root of unity, on the fact that the values of
$a_{km}(q)$, $P_{km-1}(q)$, $Q_{km-2}(q)$ and $Q_{km-1}(q)P_{km-2}(q)$
are fixed for $m$ belonging to certain arithmetic progressions (See
\eqref{cond6c}). Also important is the number $\eta$ from \eqref{cond3}
which leads to the continued fraction $H(q)$ taking only finitely
many values at roots of unity. Such $q$-continued fractions appear to
be quite special and it would interesting to have a complete
classification of them.

Our methods permit us to show the existence of a set of measure 0
at which each of the continued fractions diverges generally. We
conjecture that each of these continued fraction diverges
generally almost everywhere on the unit circle although at present
we do not see how to prove this. It would be very interesting if a
point on the unit circle which is not a root of unity could be
exhibited at which any one of the continued fractions which are
subject of Corollary \ref{cgen} converged, in either the classical
or general sense.

The most famous $q$-continued fraction after the Rogers-Ramanujan
continued fraction
is the G\"{o}llnitz-Gordon continued fraction, $GG(q)$ (see \eqref{ggcf}).
This continued fraction tends to the same limit
as $S_{2}(q)$, for each $q$ inside the unit circle, but the behaviour at
roots of unity is slightly different. As far as we are aware, its behaviour at
roots of unity has not been studied.
Based on computer investigations, it would seem that $GG(q)$
satisfies the conditions of Theorem \ref{tgen} and thus that the
G\"{o}llnitz-Gordon continued fraction  diverges at uncountably
many points on the unit circle. We hope to show this in a later paper.

{\allowdisplaybreaks

}

\end{document}